\documentclass[11pt, a4paper]{amsart}
\usepackage{amsfonts,amsmath,amssymb,amscd}
\usepackage{txfonts}
\usepackage{mathrsfs}

\usepackage[all]{xy}

\newtheorem{theorem}{Theorem}[section]
\newtheorem{lemma}[theorem]{Lemma}
\newtheorem{prop}[theorem]{Proposition}
\newtheorem{corollary}[theorem]{Corollary}

\newtheorem{fact}[theorem]{Fact}

\theoremstyle{definition}
\newtheorem{definition}[theorem]{Definition}
\newtheorem{rem}[theorem]{Remark}

\newtheorem{example}[theorem]{Example}

\usepackage{color}

\newcommand\pf{\begin{proof}}
\newcommand\epf{\end{proof}}

\newcommand{\lcosmash}{\mathop{\raisebox{0.2ex}{\makebox[0.92em][l]{${\scriptstyle>\mathrel{\mkern-4mu}\blacktriangleleft}$}}}}

\newcommand{\lboson}{\mathop{\raisebox{0.2ex}{\makebox[0.86em][l]{${\scriptstyle>\mathrel{\mkern-4mu}\lessdot}$}}\raisebox{0.12ex}{\hspace{-0.8mm}$\shortmid$}}}

\newcommand{\op}{\operatorname}
\newcommand{\cO}{\mathscr{O}}

\numberwithin{equation}{section}

\title[Affine algebraic super-groups with integral]
{Affine algebraic super-groups with integral}

\author[A.~Masuoka]{Akira Masuoka}
\address{Akira Masuoka,
Institute of Mathematics, 
University of Tsukuba, 
Ibaraki 305-8571, Japan}
\email{akira@math.tsukuba.ac.jp}

\author[T.~Shibata]{Taiki Shibata}
\address{Taiki Shibata,
Department of Applied Mathematics,\
Okayama University of Science, 
Okayama 700-0005, Japan}
\email{shibata@xmath.ous.ac.jp}

\author[Y.~Shimada]{Yuta Shimada}
\address{Yuta Shimada,
Graduate School of Pure and Applied Sciences, 
University of Tsukuba, Ibaraki 305-8571, Japan}
\email{shimada@math.tsukuba.ac.jp}

\begin{document}

\begin{abstract}
We generalize to the super context, the known fact that if an affine algebraic group $G$ over a commutative ring $k$ acts freely (in an appropriate sense) on
an affine scheme $X$ over $k$, then the dur sheaf $X\tilde{\tilde{/}}G$ of $G$-orbits is an affine scheme in the following two cases:
(I)~$G$ is finite; (II)~$k$ is a field, and $G$ is linearly reductive. 
An emphasize is put on the more difficult generalization in the second case; 
the replaced assumption then is that an affine algebraic super-group $G$ over an arbitrary
field has an integral. 
Those super-groups which satisfy the assumption are characterized,
and are seen to form a large class if $\operatorname{char}k=0$. 
Hopf-algebraic techniques including bosonization are applied to prove the results. 
\end{abstract}

\maketitle



\section{Introduction}\label{sec:introduction}

As a recent progress in super-geometry one may refer to a series of results on algebraic super-groups obtained
by using Hopf-algebraic techniques; see \cite{M1}-\cite{MZ2}, \cite{SZ0}, \cite{SZ} and \cite{Sh}. 
This paper is written so as to hopefully contribute to the progress. 

\subsection{Background and our aim}
To discuss (group) schemes and their super analogues, we adapt
the functorial point of view just as Demazure and Gabriel \cite{DG}, and
Jantzen \cite{J} do. Throughout, we work over a non-zero commutative ring $k$. Suppose that 
an affine algebraic group (scheme) $G$ (over $k$) acts freely on an affine scheme $X$ from the right. Then one can construct the 
dur sheaf $X\tilde{\tilde{/}}G$ of $G$-orbits in a simple manner. The following is known:

\begin{theorem}\label{thm:known}
$X\tilde{\tilde{/}}G$ is an affine scheme, if
\begin{itemize}
\item[(i)] 
$G$ is finite, or
\item[(ii)]
$k$ is a field, 
$G$ is linearly reductive, and the $G$-action
on $X$ is free in a stronger sense. 
\end{itemize}
\end{theorem}

The result in Case (i) is widely known as Grothendieck's Theorem. 
The result in Case (ii) was proved by Oberst; see the last part of
\cite[Satz A]{O}. 
In the situation that $k$ is a field of characteristic zero,
and $G$ is such as in Case (ii), whereas 
the $G$-action on $X$ is not assumed to be (strongly) free, Mumford et al.
\cite[Theorem 1.1]{MFK} (see also \cite[Satz D]{O}) prove the existence of
the quotient $X/G$ in the category of schemes, and that it is an affine scheme. 
Notice that the $X\tilde{\tilde{/}}G$
above is the quotient in the wider category of dur sheaves. We are more interested in 
$X\tilde{\tilde{/}}G$ when the $G$-action is free, 
since the natural morphism 
$X \to X\tilde{\tilde{/}}G$ then turns into a $G$-torsor \cite[III, $\S$4]{DG}
provided $X\tilde{\tilde{/}}G$ is an (affine) scheme.

Our aim of the paper is to generalize the known result above to the super context,
in which context torsors are called \emph{super-torsors}; see Definition \ref{DTOR}.

\subsection{Basics on super-symmetry}\label{subsec:SUSYbasics}
A \emph{super-module} $V$ over $k$ is a synonym of a $k$-module graded by the order-2 group $\mathbb{Z}_2=\{ 0,1 \}$,
and is thus the direct sum $V=V_0\oplus V_1$ of its even component $V_0$ and odd component $V_1$; 
it is called a \emph{super-vector space} if $k$ is a field. 
It is said to be \emph{purely even} if $V=V_0$. Saying 
an element $v \in V$, we often suppose, without explicit citation, that it is homogeneous, or namely, $0\ne v \in V_0\cup V_1$, and denote
its degree by $|v|$. The super-modules naturally form a tensor category, which we denote by $\mathsf{SMod}_k$. 
The morphisms are required to preserve the degree. 
The tensor product is 
the obvious one $\otimes_k$ taken over $k$, and it will be denoted by $\otimes$, simply. 
This tensor category is symmetric with respect to
the so-called \emph{super-symmetry}
\[
c_{V,W}: V \otimes W \overset{\simeq}{\longrightarrow} W \otimes V,\quad c_{V,W}(v\otimes w)=(-1)^{|v||w|}w\otimes v.
\]
A (Hopf) algebra object in $\mathsf{SMod}_k$ is called a (\emph{Hopf}) \emph{super-algebra}. 
A purely even (Hopf) super-algebra is the same as an ordinary (Hopf) algebra. 
Unless otherwise stated, (Hopf) super-algebras $R$ will be assumed to be \emph{super-commutative}. 
The assumption is precisely: (1)~the subalgebra $R_0$ of $R$ is central, and (2)~$x^2=0$ for all $x \in R_1$. If 
$2 : R \to R,\ y \mapsto 2y$ is injective (e.g., if $k$ is a field of characteristic $\ne 2$), then (1) and (2) are equivalent to that the product map
$m:R\otimes R\to R$ satisfies $m\circ c_{R,R}=m$. If $2=0$ in $R$
(e.g., if $k$ is a field of characteristic $2$), 
then they are equivalent to that $R$ 
is commutative (in the usual sense), and $x^2=0$ for all $x\in R_1$. 

\subsection{The main result}
We let $\mathsf{SAlg}_k$ denote the category of super-algebras over $k$; it is closed under the tensor product $\otimes$,
which presents the direct sum. 
Recall that the functorial point of view defines affine (group) schemes or dur sheaves to be functors which are defined
on the category $\mathsf{Alg}_k$ of commutative algebras, and satisfy certain conditions. 
We can define super-analogues of these notions,
by extending $\mathsf{Alg}_k$ to $\mathsf{SAlg}_k$; see Sections \ref{subsec:super-functor}--\ref{subsec:action} for details. 
Our main results are summarized as follows.

\begin{theorem}\label{thm:summary}
Suppose that an affine algebraic super-group $G$ acts freely on an affine super-scheme $X$. Then
the dur sheaf $X\tilde{\tilde{/}}G$ of $G$-orbits is constructed in a simple manner. 
This  $X\tilde{\tilde{/}}G$ is an affine super-scheme, 
and $X\to X\tilde{\tilde{/}}G$
is a $G$-super-torsor \textup{(}see \textup{Definition \ref{DTOR})}, if
\begin{itemize}
\item[(I)] \textup{(Theorem \ref{thm:finite})}\ $G$ is finite, or
\item[(II)] \textup{(Theorem \ref{thm:affinity_integral})}\ $k$ is a field, $G$ has an integral, and the $G$-action on $X$ is free in a stronger sense. 
\end{itemize}
\end{theorem}

For the definition of (strong) free actions see Definitions \ref{DFr} and \ref{DSFr}
below.
The result in Case (I) was obtained by Zubkov \cite{Z2}, who assumes that $k$ is a field. 
The result in Case (II) as well generalizes his result, 
since one sees that every finite super-group $G$ over a filed has an integral, and free actions by such $G$ are necessarily strongly free.

We remark that as a special case of $X\tilde{\tilde{/}}G$ as above, 
the quotient $G\tilde{\tilde{/}}H$ of an affine algebraic super-group $G$ by a closed super-subgroup $H$
is discussed in \cite{M2}, \cite{Z1}, \cite{MZ1} and \cite{MT}. 

\subsection{Our method}
Our method of proving the result in both cases is Hopf-algebraic, using the \emph{bosonization technique} developed in \cite{MZ1}.
Let $G$ and $X$ be as in the last theorem. These correspond to a Hopf super-algebra $A$ and a super-algebra $B$, respectively.
An action $X \times G \to X$ corresponds to a co-action $B \to B\otimes A$; this last is a super-algebra map, 
and hence it involves the somewhat complicated super-symmetry. 
By Radford's bosonization construction \cite{R}, there arises a co-action $\hat{B}\to \hat{B}\otimes \hat{A}$
by an ordinary Hopf algebra $\hat{A}$ on an ordinary algebra $\hat{B}$; 
this does not involve the super-symmetry anymore, but $\hat{A}$ and $\hat{B}$ are non-commutative in general. 
The bosonization technique mentioned above shows how to deduce results from the ``bosonized", ordinary situation
to the super situation. Indeed, we have a plenty of results on non-commutative Hopf algebras which deserve to be applied henceforth. 
What we need here is results in Hopf-Galois Theory, which is a non-commutative generalization of theory of torsors;
the first successful result in the Hopf-Galois Theory was probably the Kreimer-Takeuchi Theorem \cite{KT},
which is a non-commutative generalization of Grothendieck's Theorem mentioned above. 
In fact, our result above in Case (I) follows easily by applying the bosonization technique to the Kreimer-Takeuchi Theorem. 
Therefore, 
our emphasize is put on Case (II), as will be seen in view of the title of this paper. 

\subsection{The core of the paper}
Suppose that $k$ is a field. To generalize the result of Theorem \ref{thm:known} in Case (ii), one should know 
(see Remark \ref{rem:Weissauer})
that those linearly reductive affine super-groups which are not ordinary affine groups are rather restricted, as was shown by Weissauer \cite{Weissauer}.
We choose the larger class, as in (II) above, which consists of all affine algebraic super-groups $G$ with integral
(see Section \ref{subsec:preliminary_integral} for definition), to obtain the desired conclusion of Theorem \ref{thm:summary}.
Those $G$ are characterized by the property that the representation category $G$-$\mathsf{SMod}$ has enough projectives; see
Proposition \ref{prop:characterize}. 
We prove in Theorem \ref{thm:integral} 
that an affine algebraic super-group $G$ has an integral if and only if
the naturally associated affine algebraic group $G_{\mathsf{ev}}$ has an integral. 
This, combined with Sullivan's Theorem (see Theorem \ref{thm:Sullivan}) which characterizes affine algebraic groups with integral, 
tells us that the class of affine algebraic super-groups with integral is indeed large in characteristic zero; see Remark \ref{rem:large_enough}. 

We remark here that the integrals of complex affine super-groups were previously studied by 
Scheunert and Zhang \cite{SZ0}, \cite{SZ}.
It may be said that our results on integrals in characteristic zero merely refine theirs, applying subsequently developed theory. 
An advantage of ours is an explicit formula of the integral, which is applied to 
characterize connected affine algebraic super-groups $G$ with two-sided integral; 
it turns out that $G$ is such if $G_{\mathsf{ev}}$ is semisimple. See Remarks \ref{rem:Scheunert_Zhang}
and \ref{rem:explicit_integral}, 
and Proposition \ref{prop:unimodular}.  

After proving the main result, Theorem \ref{thm:affinity_integral}, in Case (II), the core section as well as the paper ends with 
Section \ref{subsec:example}; we compare there the proved result with subsequent results by Oe and the first named author \cite{MOe}, and give
an example, Example \ref{ex:alpha}, as well as
a short discussion on significance of that proved result.

\section{Affinity of quotients}\label{sec:affinity}

\subsection{Basics on super-algebras and super-modules}\label{sec:basics}
This subsection is supplementary to Section \ref{subsec:SUSYbasics}. Let $R$ be a super-algebra.
Recall that it is an algebra object in $\mathsf{SMod}_k$.
A left (resp. right) module object over $R$ is called a \emph{left} (resp., \emph{right}) $R$-\emph{super-module}.
Since $R$ is assumed to be super-commutative, we need not specify ``left" or ``right", indeed. To be more precise,
given a $k$-super-module $M$, the left and the right $R$-super-module structures on $M$ are in one-to-one 
correspondence, by twisting the side through the super-symmetry $c_{R, M} :R \otimes M\overset{\simeq}{\longrightarrow} M\otimes R$.
Let $R$-$\mathsf{SMod}$ denote the ($k$-linear abelian) category of $R$-super-modules. It has $R\oplus R[1]$ as a projective
generator, where $R[1]$ denotes the degree shift of $R$, so that $R_{i+1}=R[1]_i$, $i \in \mathbb{Z}_2$. Therefore,
an $R$-super-module is projective (in $R$-$\mathsf{SMod}$) if and only if it is a direct summand of 
some copies of $R$ or $R[1]$. Every $R$-super-module is regarded as a left and right module over the algebra $R$,
with the $\mathbb{Z}_2$-grading forgotten. 

\begin{lemma}\label{lem:basic_super-module}
Let $M$ be an $R$-super-module.
\begin{itemize}
\item[(1)] If $M$ is finitely generated as a left or right $R$-module, then it has a finite set of homogeneous generators.
\item[(2)] $M$ is \textup{(}faithfully\textup{)} flat as a left or/and right $R$-module if and only if the tensor-product functor
\[ M \otimes_R : R\text{-}\mathsf{SMod} \to R\text{-}\mathsf{SMod} \]
is \textup{(}faithfully\textup{)} exact. 
\item[(3)]
$M$ is projective in $R$-$\mathsf{SMod}$ if and only if it is projective as a left or/and right $R$-module.
\end{itemize}
\end{lemma}
\pf
(1)\ This is easy to see.\
(2)\ See the proof of \cite[Lemma 5.1 (1)]{M1}, which works as well when the base $k$ is a commutative ring.\
(3)\ See the proof of \cite[Lemma 5.1 (2)]{M1}, which is based on the fact that
the smash product $R \rtimes (k\mathbb{Z}_2)^*$ is separable over $R$. 
\epf

\begin{lemma}\label{lem:generator}
Let $R$ be a super-algebra, and let $P$ be an  
$R$-super-module which is finitely generated projective as a \textup{(}left or right\textup{)} $R$-module. 
Then $P$ is a generator of $R$-modules 
if and only if $P$ is faithful as an $R$-module.
\end{lemma}
\begin{proof}
``Only if" is obvious. 
For ``if'', assume that $P$ is faithful. 
To prove that the trace ideal $\mathfrak{T}(P)$ of $P$ coincides
with $R$, we wish to show that the localizations 
$\mathfrak{T}(P)_{\mathfrak{m}}$ and $R_{\mathfrak{m}}$ at every maximal
ideal $\mathfrak{m}$ of the central subalgebra $R_0 \subset R$ coincide.  
Since $R_{\mathfrak{m}}$ is local (indeed, $R_{\mathfrak{m}}/\mathfrak{m}R_{\mathfrak{m}}$ is a field modulo the nil ideal generated
by the odd component),
$P_{\mathfrak{m}}$ is a finitely generated free
$R_{\mathfrak{m}}$-module, which has $\mathfrak{T}(P)_{\mathfrak{m}}$ as its trace ideal. 
It remains to prove $P_{\mathfrak{m}}\ne 0$. 

An element of $R$ annihilates $P$ if and only if it
annihilates all elements in an arbitrarily chosen set of homogeneous generators
of $P$; see Lemma \ref{lem:basic_super-module} (1).
Therefore, 
the faithfulness of $P$ is expressed as an $R$-super-module injection $R \to P_1\oplus \dots \oplus P_r$,
where $0<r<\infty$, and each $P_i$ is a copy of $P$ or of its degree shift $P[1]$. 
This implies $P_{\mathfrak{m}}\ne 0$. 
\end{proof}

Suppose that $k$ is a field. 
Let $C$ be a super-coalgebra, or namely, a coalgebra object in $\mathsf{SMod}_k$.
A (left or right) $C$-\emph{super-comodule} is a $C$-comodule object. 

\begin{lemma}
A left or right $C$-super-comodule is injective in the category of those super-comodules 
if and only if it is injective, regarded as an ordinary comodule
over the coalgebra $C$. 
\end{lemma}

\begin{proof}
This follows easily by dualizing the proof of Lemma \ref{lem:basic_super-module} (3), above. 
\end{proof}


\subsection{Basics on super-functors}\label{subsec:super-functor}
Suppose that $R \to S$ is a map of super-algebras. It is said to be  
an \emph{fpqc covering}, 
if $S$ is faithfully flat as a left or equivalently, right $R$-module; see Lemma \ref{lem:basic_super-module} (2). 
The map is said to be an \emph{fppf covering}, if it is an fpqc covering, and if
$S$ is finitely presented as a super-algebra over $R$; the second assumption means that $S$ is presented
so as $R[x_1,\cdots,x_m;y_1,\cdots,y_n]/I$, where $x_1,\dots,x_m$ are finitely many even variables, $y_1,\dots,y_n$
are finitely many odd variables, and $I$ is a super-ideal which is generated by finitely many homogeneous elements. 

Recall that $\mathsf{SAlg}_{k}$ denotes the category of super-algebras over $k$. 
A $k$-\emph{super-functor} is a set-valued functor defined on the category 
$\mathsf{SAlg}_{k}$.
A $k$-super-functor $X$ is called an \emph{affine super-scheme} (over $k$), 
if it is representable, and in addition, if the super-algebra which represents $X$ is non-zero;
the added assumption is equivalent to that $X(R)\ne \emptyset$ for some $0\ne R\in \mathsf{SAlg}_{k}$.
We say that a $k$-super-functor $X$ is
a \emph{dur sheaf} (resp., \emph{sheaf}), if it preserves finite direct products and the exact diagram
\[ R \to S \rightrightarrows S \otimes_R S \]
which naturally arises from any fpqc (resp., fppf) covering $R\to S$. The dur sheaves and the 
sheaves form full subcategories in the category of $k$-super-functors; the latter subcategory includes
the former, which in turn includes the category of affine super-schemes.

An \emph{affine super-group} (over $k$) is an representable group-valued functor defined  
on $\mathsf{SAlg}_{k}$; it is uniquely represented by a Hopf super-algebra.
Given an affine super-scheme or super-group $X$, we let $\cO(X)$ denote the (Hopf) super-algebra 
which represents $X$. In this case, $X$ is said to be \emph{algebraic} (resp., \emph{Noetherian}),
if $\cO(X)$ is finitely generated
as an algebra (resp., if $\cO(X)$ is Noetherian \cite[Section A.1]{MZ2} in the sense that its
super-ideals satisfy the ascending chain condition). 

\subsection{Actions by affine super-groups}\label{subsec:action}
Let $X$ be an affine super-scheme, and let $G$ be an affine super-group. 
Suppose that $G$ acts on $X$. Here and in what follows, $G$-actions are always supposed to be from the right. 
Thus we are given a super-algebra map 
\begin{equation}\label{eq:rho}
\rho : \cO(X) \to \cO(X)\otimes \cO(G)
\end{equation}
by which $\cO(X)$ is a right $\cO(G)$-comodule.
The \emph{super-subalgebra of $G$-invariants} (or \emph{of $\cO(G)$-co-invariants}) in $\cO(X)$ is defined by 
\begin{equation}\label{eq:invariants}
\cO(X)^{G}= \{ b \in \cO(X) \mid \rho(b) = b \otimes 1\}.
\end{equation}
This is indeed a super-subalgebra of $\cO(X)$. 

Let us consider the morphism
\begin{equation}\label{eq:dual_alpha}
X \times G \to X \times X,\quad (x,g) \mapsto (x, xg) 
\end{equation}
of $k$-super-functors.

\begin{definition}\label{DFr}
We say that the $G$-action on $X$ is \emph{free}, or $G$ acts \emph{freely} on $X$,  
if for every super-algebra $R$, 
the map $X(R) \times G(R) \to X(R) \times X(R)$ given as above
is injective; see \cite[Part I, 5.5]{J}. The condition is equivalent to saying that
the morphism \eqref{eq:dual_alpha} is a monomorphism in the category 
of (affine) super-schemes. 
\end{definition}

Obviously, the morphism \eqref{eq:dual_alpha} is represented by 
the super-algebra map
\begin{equation}\label{eq:alpha}
\cO(X)\otimes \cO(X) \to \cO(X) \otimes \cO(G),\quad
a \otimes b\mapsto a \, \rho(b). 
\end{equation}

\begin{definition}\label{DSFr}
We say that the $G$-action on $X$ is \emph{strongly free}, or $G$ acts 
\emph{strongly freely} on $X$,
if the super-algebra map \eqref{eq:alpha} is surjective. The condition
is equivalent to saying that the morphism \eqref{eq:dual_alpha}, regarded
as that of (affine) super-schemes, is a closed immersion. 
\end{definition}

Mumford et al. \cite[Definition 0.8]{MFK} and Oberst \cite[p.510]{O}
call strongly free actions  
(in our sense) \emph{free}, 
in the ordinary, or namely, non-super situation.  

\begin{prop}\label{PFree}
The $G$-action on $X$ is free, if it is strongly free. The converse holds provided
$\cO(G)$ is finitely generated as a $k$-module
\end{prop}
\pf 
The first assertion is obvious. 
The second follows, 
since 
a map $f:R\to S$ of super-rings 
(or namely, of super-algebras over $\mathbb{Z}$)
such that 
$S$ is finitely generated, regarded as an $R$-module through the map, is surjective, 
if and only if it is an epimorphism in the category of super-rings. 
The essential ``only if" part follows as in the non-super situation 
by using the fact that
$f:R \to S$ is an epimorphism in the category if and only if the map tensored with the identity map on $S$
\begin{equation*}\label{eq:SSRS}
f\otimes_R \mathrm{id}:S=R\otimes_RS\to S\otimes_RS
\end{equation*}
is bijective. In fact, if $f$ is not surjective,
then the cokernel $\operatorname{Coker} f$, being finitely generated, has a non-zero quotient (cyclic) $R$-super-module 
$R/I$, where $I\subsetneqq R$ is a super-ideal. 
Since $(\operatorname{Coker} f)\otimes_RS\,
 (=\operatorname{Coker}(f\otimes_R\mathrm{id}))$, having $(R/I)\otimes_R(R/I)\, (=R/I)$ 
as a non-zero quotient, is non-zero, the map $f\otimes_R\mathrm{id}$ cannot be bijective.
\epf

Suppose that an affine super-group $G$ acts freely on an affine super-scheme $X$.
We have the $k$-super-functor which associates to each super-algebra $R$, 
the set $X(R)/G(R)$ of all $G(R)$-orbits in $X(R)$. The freeness assumption 
ensures that the map $X(R)/G(R)\to X(S)/G(S)$ is injective if $R \to S$ is injective. 
This makes it possible, as shown in \cite[Part I, 5.4]{J} in the non-super situation, to construct in a simple manner,
 the dur sheaf $X \tilde{\tilde{/}} G$ and the sheaf $X \tilde{/} G$  
which both are associated with the $k$-super-functor above. 
This dur sheaf (resp., sheaf) 
is characterized by the exact diagram
\begin{equation}\label{EEXACT}
X \times G \rightrightarrows X \to X \tilde{\tilde{/}} G
\quad (\text{resp.,}\ X \times G \rightrightarrows X \to X \tilde{/} G)
\end{equation}
in the category of dur sheaves (resp., of sheaves), where the paired arrows represent
the action and the projection.

\subsection{Affinity criteria}\label{subsec:affinity_criteria}
We reproduce from \cite{MZ1} the following super-analogue of Oberst's Theorem \cite[Satz A]{O} in a slightly restricted form so as to meet our use;
it originally excluded the case when $\operatorname{char}k= 2$, in which case the proof is seen to work as well,
under the present, appropriate definition of super-commutativity. 

\begin{theorem}[\text{\cite[Theorem 7.1]{MZ1}}]\label{thm:affinity}
Suppose that $k$ is a field. 
Suppose that an affine super-group $G$ acts freely
on an affine super-scheme $X$. 
\begin{itemize}
\item[(1)]
The following are equivalent:
\begin{itemize}
\item[(a)]
The dur sheaf $X \tilde{\tilde{/}} G$ is an affine super-scheme;
\item[(b)]
The super-algebra map given by \eqref{eq:alpha} is surjective \textup{(}or in other words,
the $G$-action on $X$ is strongly free\textup{)}, and
$\cO(X)$ is injective as a right $\cO(G)$-comodule;
\item[(c)]
$\cO(X)^{G}\hookrightarrow \cO(X)$ is an fpqc covering, and
the super-algebra map 
\begin{equation}\label{eq:beta}
\beta : \cO(X) \otimes_{\cO(X)^{G}} \cO(X) \to \cO(X) \otimes \cO(G),\quad 
\beta(a \otimes b) = a \rho (b)
\end{equation}
is bijective.
\end{itemize}
If these conditions are satisfied, then $X \tilde{\tilde{/}} G$ is represented by $\cO(X)^{G}$.
\item[(2)]
Suppose that $X$ is Noetherian, and $G$ is algebraic. 
If the equivalent conditions above are satisfied, then 
the affine super-scheme $X \tilde{\tilde{/}} G$
is Noetherian, and coincides with the sheaf $X \tilde{/} G$. 
\end{itemize}
\end{theorem}

\begin{rem}\label{rem:over_ring}
Suppose that $k$ is a non-zero commutative ring, and an affine super-group $G$ 
acts freely on an affine super-scheme $X$. 
Let $Y$ denote the affine super-scheme represented by $\cO(X)^{G}$. 
Then Conditions (a) and (c) above remain to be equivalent. Moreover,
if they are satisfied, then $X \tilde{\tilde{/}} G=Y$. To see this, assume (a), first.
The same argument as proving \cite[Part I, 5.7, (1)]{J} shows that the natural morphism 
$X \to X \tilde{\tilde{/}} G$ is faithfully flat
(or namely, $\cO(X) \hookleftarrow \cO(X \tilde{\tilde{/}} G)$ 
is an fpqc covering), and the morphism of $k$-super-functors
\begin{equation}\label{eq:XG_XYX}
X\times G \to X \times_{X\tilde{\tilde{/}}G} X,\quad (x,g)\mapsto (x,xg)
\end{equation}
is isomorphic. This implies 
(c) together with $X \tilde{\tilde{/}} G=Y$. Next, assume (c).
With every dur sheaf applied to the exact diagram of super-algebras
\[ \cO(Y) \to \cO(X) \rightrightarrows \cO(X) \otimes_{\cO(Y)} \cO(X)\ (\simeq \cO(X\times G)) \]
associated with the fpqc covering $\cO(Y)\hookrightarrow \cO(X)$, the characterization 
of $X \tilde{\tilde{/}} G$ given by \eqref{EEXACT} shows 
$X \tilde{\tilde{/}} G=Y$, which ensures (a). 

In the same situation as in (2) above, we see also that if the equivalent Conditions (a) and (c)
are satisfied, then $X \tilde{\tilde{/}} G$, $X \tilde{/} G$ and $Y$ coincide, and they are
Noetherian. 
\end{rem}

In the situation of Remark \ref{rem:over_ring} above
(in particular, over a non-zero commutative ring $k$),
assume that the equivalent Conditions (a) and (c) are satisfied. 
Then $X \to Y\, (=X \tilde{\tilde{/}} G)$ is a faithfully flat morphism
of affine super-schemes, 
and \eqref{eq:XG_XYX} is isomorphic. 

\begin{definition}\label{DTOR}
In this case we say that $X\to Y$
is a $G$-super-torsor.
\end{definition}

This may be alternatively called an \emph{algebraic
principal super-bundle} with structure super-group $G$.
The term ``super-torsor" comes from \cite[III. $\S$4]{DG}, 
which defines the notion in a rather
generalized situation of the non-super setting. 

\begin{rem}\label{rem:torsor}
$G$-super-torsors are a subject of interest for further study, too. 
Some subsequent results by Oe and the first-named author \cite{MOe} will be presented without proof in Section \ref{subsec:example},
to compare with our main result, Theorem \ref{thm:affinity_integral}.
We remark that the above-given definition of super-torsors is restricted 
to the situation where $X$ and $Y$ are affine super-schemes; in fact, the 
subject is discussed in the present paper only in the restricted situation,
whereas in the just cited \cite{MOe} generally when $X$ and $Y$ are 
(not necessarily affine) super-schemes.
\end{rem}

\subsection{Bosonization technique}\label{subsec:boson_tech}
Here we recall from \cite[Section 10]{MZ1} the technique, which was used to prove 
Theorem \ref{thm:affinity} above, and will play a role in the sequel. 

Suppose that $k$ is a non-zero commutative ring. 
Present $\mathbb{Z}_2$ as a multiplicative group so as
\begin{equation}\label{eq:Z2}
\mathbb{Z}_2 = \langle \sigma \mid \sigma^2=e \rangle. 
\end{equation}
Given a super-algebra $A$ which may not be super-commutative, 
let the generator $\sigma$ of $\mathbb{Z}_2$ act on $A$ so that
\[ \sigma \cdot a = (-1)^{|a|} a,\quad a \in A, \]
and let $A \rtimes \mathbb{Z}_2$ denote the resulting algebra of semi-direct (or smash) product. 
Dually, if $A=(A,\delta,\epsilon)$ is a super-coalgebra, then the coalgebra $A \lcosmash \mathbb{Z}_2$
of smash co-product is constructed on the tensor product $A\otimes k \mathbb{Z}_2$ by 
the structure maps
\[
\Delta(a\otimes \sigma^i) = \sum_j (b_j \otimes \sigma^{i+|c_j|}) \otimes (c_j\otimes \sigma^i),\quad
\varepsilon(a \otimes \sigma^i)=\epsilon(a), 
\]
where $i \in \{ 0,1\}$,\ $a \in A$ and $\delta(a)=\sum_jb_j\otimes c_j$. 

Suppose that $A$ is a Hopf super-algebra which may not be super-commutative. 
The \emph{bosonization} $A \lboson \mathbb{Z}_2$ 
of $A$ is the ordinary Hopf algebra defined on $A \otimes k \mathbb{Z}_2$
which is $A \rtimes \mathbb{Z}_2$ as an algebra,
and is $A \lcosmash \mathbb{Z}_2$ as a coalgebra. The antipode $S$ is given by 
\[
S(a \otimes \sigma^i)= (-1)^{|a|(i+1)}s(a)\otimes \sigma^{i+|a|},\quad a \in A,\ i \in \{ 0,1\},
\]
where $s$ is the antipode of $A$. This $S$ is bijective if $s$ is; this is the case if $A$ is super-commutative, and hence
$s$ is an involution. But even under the assumption, $S$ is not necessarily an involution. 
(A historical remark: the construction was originally done by Radford \cite{R} under the name ``bi-product" in the generalized situation
where $\mathbb{Z}_2$, or the group algebra $k \mathbb{Z}_2$, is replaced by an arbitrary Hopf algebra.) 

Suppose that we are in the situation of Section \ref{subsec:action}, so that  
an affine super-group $G$ acts on an affine super-scheme $X$. 
We set
\[
A=\cO(G),\quad B=\cO(X),\quad C= B^G,\quad 
\hat{A}=A \lboson \mathbb{Z}_2,\quad \hat{B}=B \rtimes \mathbb{Z}_2. 
\]
The structure map $\rho : B \to B \otimes A$ given in \eqref{eq:rho} gives rise to the algebra
map
\[
\hat{\rho} : \hat{B}\to \hat{B}\otimes \hat{A},\quad 
\hat{\rho}(b\otimes \sigma^i)= \sum_j (b_j \otimes \sigma^{i+|a_j|})\otimes (a_j\otimes \sigma^i), 
\]
where $i \in \{0,1\}$, $b \in B$ and $\rho(b)=\sum_j b_j \otimes a_j$. Moreover, $(\hat{B},\hat{\rho})$ 
is a right $\hat{A}$-comodule (algebra), and the subalgebra 
\[
\hat{B}^{\mathsf{co}\hspace{0.3mm}\hat{A}}=\{ b \in \hat{B}\mid \hat{\rho}(b)=b \otimes 1\}
\]
of $\hat{A}$-co-invariants in $\hat{B}$
coincides with $C\, (=C\otimes k)$. This and the following are proved in \cite[Proposition 10.3]{MZ1}
in the generalized situation noted above. 

\begin{lemma}\label{lem:bijective_beta}
The map $\beta : B \otimes_C B \to B \otimes A$ defined in \eqref{eq:beta} is surjective \textup{(}resp., bijective\textup{)}
if and only if the map
\[
\hat{\beta} : \hat{B} \otimes_C \hat{B}\to \hat{B}\otimes \hat{A},\quad
\hat{\beta}(a \otimes b)=a\, \hat{\rho}(b) 
\]
surjective \textup{(}resp., bijective\textup{)}. 
\end{lemma}

This lemma will be used to deduce new results in the super situation from  
known ones on ordinary Hopf algebras, which include some important ones by Schauenburg and Schneider \cite{SS}, in particular.

\subsection{Quotients by finite super-groups}\label{subsec:finite_supergroup}

Suppose that $k$ is a non-zero commutative ring. An affine super-group $G$ is said to be \emph{finite}
if $\cO(G)$ is finitely generated projective as a $k$-module. The following theorem, a super-analogue
of Grothendieck's Theorem, generalizes 
Theorem 0.1 of Zubkov \cite{Z2}, who assumes that $k$ is a field. 
Our proof using the bosonization technique is different from and simpler than the one in \cite{Z2}. 

\begin{theorem}\label{thm:finite}
Suppose that a finite affine super-group $G$ acts freely on an affine super-scheme on $X$. 
Then $\cO(X)^G \hookrightarrow \cO(X)$ is an fppf covering. The dur sheaf $X \tilde{\tilde{/}} G$
and the sheaf $X \tilde{/} G$ coincide, and they are in fact the affine super-scheme represented by $\cO(X)^{G}$. 
Moreover, $X\to X \tilde{\tilde{/}} G\, (=X \tilde{/} G)$
is a $G$-super-torsor.
\end{theorem}

\begin{proof}
Let us use the same notation as in the preceding subsection. By Remark \ref{rem:over_ring} 
it suffices to prove that 
(i)~$C \hookrightarrow B$ is an fppf covering, and (ii)~$\beta$ is bijective. 

Notice from the assumptions and Proposition \ref{PFree} that $\beta$ is surjective, 
whence $\hat{\beta}$ is, too. Moreover,
the Hopf algebra $\hat{A}$ is finitely generated projective as a $k$-module. 
It follows by the Kreimer-Takeuchi
Theorem \cite{KT} that 
(iii)~$\hat{B}$ is finitely generated projective, or equivalently, finitely presented flat, as a (left and right) $C$-module, and 
(iv)~$\hat{\beta}$ is bijective.
By Lemma \ref{lem:bijective_beta}, (iv) ensures (ii). 

By Lemma \ref{lem:generator}, $C \ \hookrightarrow B$ splits $C$-linearly.
This, combined with the flatness of (iii), shows that $C \hookrightarrow B$ is an fpqc covering.  
The remaining (i) follows from the ``only if" part of the following (cf.~\cite[Remark 1.1]{Z2}):

\begin{fact}
In general, given a map $C \to B$ of super-rings,
$B$ is finitely presented as a $C$-module if and only if $B$ is finitely generated as a $C$-module and finitely presented 
as a $C$-super-algebra. 
\end{fact} 

This can be proved essentially in the same way as in the non-super situation;
see \cite[Proof of (B12), p.570]{GW}, for example. 
In fact, the now needed ``only if" is proved as follows
(in a slightly different 
manner from the one of the article cited above).
Assume that $B$ is finitely presented as a $C$-module.
Then we have an exact sequence 
\[
P\overset{f}{\longrightarrow} Q \overset{g}{\longrightarrow} B \to 0
\]
of $C$-super-modules, where $P$ and $Q$ are direct
sums of some finitely many copies of $C$ or $C[1]$. 
Let $b_q$ be the elements of 
$B$ which are the images of the canonical $C$-free basis elements of $Q$ under $g$.
Let $C'=\mathbb{Z}[c_{pq}, d_{qr}^{s}, e_q]$ denote the $\mathbb{Z}$-super-subalgebra of $C$ generated by 
all $c_{pq}, d_{qr}^s, e_q$ (finitely many homogeneous elements), where $c_{pq}$ are the entries
of the matrix which presents $f:P\to Q$ with respect to the canonical $C$-free bases, and
$d_{qr}^s$ and $e_q$ are elements of $C$ chosen so that they satisfy
\[
b_qb_r=\sum_{s}d_{qr}^sb_s,\quad \sum_q e_q b_q=1.
\]
Define $B':=\sum_qC'b_q$ in $B$. 
Then $B'$ is a finitely generated super-algebra over the Noetherian super-algebra 
$C'$, which is, therefore, finitely presented. We have a natural surjection 
\[
h : B'\otimes_{C'} C \to B
\]
of $C$-super-algebras. We claim that this $h$ is an isomorphism; the claim implies
that $B$ is finitely presented as a $C$-super-algebra, as desired. The map $f$ is the base extension 
of a canonical $C'$-form $f' :P'\to Q'$, which, composed with the natural surjection $Q'\to B'$, turns into a zero map.
The induced  (surjective) map $\operatorname{Coker}(f')\to B'$, after base extension,
turns into $\operatorname{Coker}(f')\otimes_{C'}C\to B'\otimes_{C'}C$,
which turns, composed with $h$, to coincide with a canonical isomorphism $\operatorname{Coker}(f')\otimes_{C'}C\overset{\simeq}{\longrightarrow} B$.
This implies the claim. 
\end{proof}

\section{Integrals for affine super-groups}\label{sec:integral}

Throughout in this section we suppose that $k$ is a field, and let $G$ be an affine super-group
over $k$; it will be often assumed to be algebraic. 
We let $\bar{k}$ denote the algebraic closure of $k$, 
and $G_{\bar{k}}$ denotes the base extension of $G$ to $\bar{k}$.

\subsection{Preliminary results}\label{subsec:preliminary_integral}
We let
\begin{equation}\label{eq:A}
A := \cO(G)
\end{equation}
denote the Hopf super-algebra which represents $G$, as before. 
Following \cite{SZ}, a \emph{left integral} for $G$ is defined to be an element
$\phi$ in the dual algebra $A^*$ of the coalgebra $A$, such that
\[ \xi \, \phi = \xi(1)\, \phi,\quad \xi \in A^*. \]
Such an element is identified with a left (not necessarily $\mathbb{Z}_2$-graded) $A$-comodule map 
$\phi : A \to k$,
where $k$ is regarded as a trivial $A$-comodule. A \emph{right integral} for $G$ is defined analogously.  
The left integrals and  the right integrals
are in one-to-one correspondence, through the dual $s^*$ of the antipode $s$ of $A$.
Let $\hat{A}=A \lboson \mathbb{Z}_2$ denote
the Hopf algebra of bosonization constructed in
Section \ref{subsec:boson_tech}. 

In general, given a coalgebra $C$, we let
\begin{equation}\label{eq:HomC}
\operatorname{Hom}^{-C}(C, k)
\end{equation}
denote the vector space of right $C$-comodule maps $C\to k$. We now use this notation
when $C=A$\ or $\hat{A}$. 

\begin{prop}[\text{\cite[Theorem 1]{SZ0}}]\label{prop:integral_boson}
There is a natural $k$-linear isomorphism
\[
\operatorname{Hom}^{-A}(A, k)\simeq \operatorname{Hom}^{-\hat{A}}(\hat{A}, k).
\]
\end{prop}
\pf
We suppose $\mathbb{Z}_2=\langle \sigma\mid \sigma^2=e \rangle$ as in \eqref{eq:Z2}, and define 
\[
\omega :k \mathbb{Z}_2\to k, \quad \omega(\sigma^i)=\delta_{i,0}, \ i\in \{0,1\}. 
\] 
This $\omega$ is a unique (up to scalar multiplication) right $k \mathbb{Z}_2$-comodule map; it is necessarily a right $\hat{A}$-comodule map
since $k\mathbb{Z}_2$ is a Hopf subalgebra of $\hat{A}$.  We add the remark: by the same reason every right $\hat{A}$-comodule
map $\phi : \hat{A} \to k$ vanishes on $A\otimes \sigma$ in $\hat{A}=A\otimes k \mathbb{Z}_2$. Indeed, given $a \in A$, 
the map $k\mathbb{Z}_2\to k$, $x \mapsto \phi(a\otimes x)$ is a right $k\mathbb{Z}_2$-comodule map, whence it vanishes at $\sigma$. 

Given $\psi \in \operatorname{Hom}^{-A}(A, k)$, define $\phi : \hat{A}=A\otimes k\mathbb{Z}_2\to k$ by
$\phi:=\psi\otimes \omega$. This $\phi$ coincides with the composite
\[ \hat{A}=A\square_A\hat{A} \overset{\psi \square_A \mathrm{id}_{\hat{A}}}{\longrightarrow} k \square_A\hat{A}
=k\mathbb{Z}_2\overset{\omega}{\longrightarrow} k\]
of the co-tensor product $\psi \square_A \mathrm{id}_{\hat{A}}$
(see \cite[Section 2.3]{DNR}) with $\omega$. 
Hence we have $\phi \in \operatorname{Hom}^{-\hat{A}}(\hat{A}, k)$. 
The assignment $\psi \mapsto \phi$ gives a desired isomorphism.
Indeed, one sees by using the remark above that the inverse assigns to each $\phi$, the map $A \to k, a \mapsto 
\phi(a\otimes e)\, (=\phi(a\otimes (e+\sigma)))$. 
\epf

Clearly, the proposition holds for any arbitrary Hopf super-algebra that may not be super-commutative,
and the result is essentially shown by Scheunert and Zhang \cite{SZ0} by essentially the same method. 
The proof of ours, which uses the co-tensor product, might be slightly simpler. 

\begin{corollary}
A non-zero left or right integral for $G$, 
if it exists, is unique up to scalar multiplication.   
\end{corollary}
\pf
This follows by the last proposition combined with the well-known uniqueness for ordinary Hopf
algebras \cite{Su1} due to Sullivan; see also \cite[Theorem 5.4.2]{DNR}. 
\epf

It follows that a non-zero left or right integral for $G$ is homogeneous. To be more explicit, 
the homogeneous components $\phi_0$, $\phi_1$ of a (left or right) integral $\phi \in A^*$ 
are integrals, whence $\phi=\phi_0$ or $\phi=\phi_1$. This means that
the $A$-comodule map $\phi : A\to k$ vanishes on $A_1$ or on $A_0$; see \cite[Theorem 1]{SZ0}, again.

\begin{example}[\text{\cite[Example 1]{SZ0}}]\label{ex:exterior}
Let $W$ be a vector space of dimension $n<\infty$, and let $\wedge(W)$ denote the exterior 
algebra on $W$, which we regard as a (super-commutative and super-cocommutative) Hopf super-algebra
with all elements $w$ in $W$ odd primitives, $\Delta(w)=1\otimes w+w\otimes 1,\ \varepsilon(w)=0$. 
Choose arbitrarily a basis $w_1, w_2,\dots,w_n$ of $W$. 
Let
\begin{equation}\label{eq:increasing_sequence}
\Lambda_n:=\{\, I=(i_1,i_2,\dots,i_r)\mid \ 0\le i_1<i_2<\dots <i_r\le n,\ 0\le r \le n \, \}
\end{equation}
denote the set of all strictly increasing sequences of positive integers $\le n$. 
Then this $\wedge(W)$ has 
\[ w_I=w_{i_1}\wedge w_{i_2}\wedge\dots\wedge w_{i_r},\quad I=(i_1,i_2,\dots,i_r)\in \Lambda_n \]
as a basis. One sees that the $k$-linear map
\begin{equation*}
\phi : \wedge(W)\to k,\quad \phi(w_I)=
\begin{cases} 1 &\text{if}\ \, I=(1,2,\dots,n); \\ 
0 &\text{otherwise}\end{cases}
\end{equation*}
is a non-zero left and right integral for the finite affine super-group represented by $\wedge(W)$,
which is even (resp., odd) if $n\, (=\dim W)$ is even (resp., odd). 
\end{example}

\begin{definition}\label{def:G_with_integral}
We say that $G$ \emph{has an integral} or $G$ is an \emph{affine super-group with integral}, 
if $G$ has a non-zero left or right 
integral. This is equivalent to saying that $A\, (=\cO(G))$ is (left or/and right) co-Frobenius as a coalgebra; 
see \cite[Section 5.3]{DNR}. 
We say that $G$ is \emph{unimodular}, 
if it has a non-zero, left and at the same time right 
integral. 
\end{definition}


Left and right $A$-super-comodules are naturally identified. To be more precise, given a super-vector space $V$, 
the left and the right $A$-super-comodule structures on $V$ are in one-to-one correspondence, by twisting the side
through 
$V \otimes A\overset{\operatorname{id}\otimes s}{\longrightarrow}V\otimes A\overset{c_{V,A}}{\longrightarrow}A\otimes V$,
where $s$ denotes the antipode of $A$.  
A left (resp., right) $A$-super-comodule is naturally identified with a left (resp., right) $\hat{A}$-comodule.
One may understand that a \emph{left} (resp., \emph{right}) $G$-\emph{super-module} is by definition a right (resp., left) $A$-super-comodule. 
We may and do choose \emph{left} $G$-super-modules, and denote their category by $G$-$\mathsf{SMod}$. 
This is a $k$-linear abelian, symmetric category with enough injectives. In view of Proposition \ref{prop:integral_boson} the next proposition follows
by applying to our $\hat{A}$, the characterization \cite[Theorem 3.2.3]{DNR} for a coalgebra to be co-Frobenius.

\begin{prop}\label{prop:characterize}
For an affine super-group $G$, the following are equivalent:
\begin{itemize}
\item[(a)] $G$ has an integral;
\item[(b)] $G$-$\mathsf{SMod}$ has enough projectives; 
\item[(c)] Every injective object in $G$-$\mathsf{SMod}$ is projective;
\item[(d)] For every finite-dimensional object in $G$-$\mathsf{SMod}$, its injective hull is finite-dimensional.
\end{itemize}
\end{prop}

\begin{rem}\label{rem:Weissauer}
A left or right integral $\phi$ for an affine super-group $G$ 
is said to be \emph{total} if $\phi(1) = 1$. Such an integral, if it exists,
is unique (in the strict sense), and is a left and right integral, as is easily seen; see \cite[Proposition 2]{SZ0}. 
We say that $G$ is \emph{linearly reductive}, if it has a total integral.  
This is the case if and only if every object in $G$-$\mathsf{SMod}$ is semisimple
if and only if the coalgebra $A$ (or equivalently, $\hat{A}$) is \emph{cosemisimple} \cite[p.199]{DNR}.  
As was shown 
by Weissauer \cite{Weissauer}, those linearly reductive affine super-groups which are not purely even are rather restricted
even in characteristic zero. 
\end{rem}

The ordinary (commutative) algebras,
regarded as purely even super-algebras, form a full subcategory, $\mathsf{Alg}_{k}$, of 
$\mathsf{SAlg}_{k}$. The restricted group-valued functor $G|_{\mathsf{Alg}_{k}}$, 
which we denote by $G_{\mathsf{ev}}$,
is an affine group, which is represented by the largest purely even quotient Hopf super-algebra 
\begin{equation}\label{eq:H}
H :=A/(A_1)
\end{equation}
of $A$. Here $(A_1)$ denotes the super-ideal generated by the odd component $A_1$ of $A$. 
Let $W=T^*_e(G)_1$ denote the odd component of the cotangent super-vector space
$T^*_e(G)$ of $G$ at the identity element $e$; we have $\dim W<\infty$ and $G_{\mathsf{ev}}$
is algebraic, provided $G$ is algebraic.  
For the exterior algebra on $W$, recall 
from Example \ref{ex:exterior} only the co-unit $\varepsilon:\wedge(W)\to k;\ \varepsilon(w)=0,\ w\in W$. 
By \cite[Theorem 4.5]{M1} (see also \cite[Theorem 5.7]{MS2})
there exists a co-unit-preserving isomorphism
\begin{equation}\label{eq:isom}
A\simeq \wedge(W)\otimes H
\end{equation} 
of \emph{right $H$-super-comodule algebras}, i.e., algebra objects in the tensor category of right $H$-super-comodules. 

\subsection{Affine algebraic super-groups with integral }\label{subsec:integral}
In this subsection we aim to prove the following theorem, which was stated as
Proposition 7.5 in \cite{M3}
without proof.

\begin{theorem}\label{thm:integral}
Suppose that $G$ is an affine algebraic super-group; $G_{ev}$ is then an affine algebraic group. 
The following are equivalent:
\begin{itemize}
\item[(i)] $G$ has an integral; 
\item[(ii)] $G_{\mathsf{ev}}$ has an integral. 
\end{itemize}
\end{theorem}

\begin{proof}[Proof of (i) $\Rightarrow$ (ii)]
The implication follows since the right $\mathscr{O}(G_{\mathsf{ev}})$-comodule $\mathscr{O}(G)$ 
is \emph{co-free} (i.e., the direct sum of some copies of $\mathscr{O}(G_{\mathsf{ev}})$), as is seen from \eqref{eq:isom}. 
\end{proof}

A proof of (ii) $\Rightarrow$ (i) will be given in Sections 
\ref{subsubsec:proof_in_char_p}--\ref{subsubsec:proof_in_char_zero} below. 

The following Sullivan's Theorem for $F$ applied to $G_{\mathsf{ev}}$ tells us 
precisely when Condition (ii) above is satisfied.

\begin{theorem}[\text{Sullivan}]\label{thm:Sullivan} 
Let $F$ be an affine algebraic group. 
\begin{itemize}
\item[(1)] Suppose $\op{char}k=0$. Then $F$ has an integral if and only if $F$ is linearly reductive.
\item[(2)] Suppose $\op{char}k>0$, and let $F^{\circ}$ be the 
connected component of $F$ containing the identity element. 
Then $F$ has an integral if and only if the reduced affine algebraic
$\bar{k}$-group $(F^{\circ}_{\bar{k}})_{\mathsf{red}}$
associated with the base extension $F^{\circ}_{\bar{k}}$ of $F^{\circ}$ to
$\bar{k}$ is a torus. 
\end{itemize}
\end{theorem}

\begin{rem}\label{rem:large_enough}
Suppose $\op{char}\, k=0$. One sees from the preceding two theorems that the
connected (see below) affine algebraic super-groups with integral are precisely 
what Serganova \cite{Se} proposed to call \emph{quasi-reductive super-groups};
they form a large class which includes 
\emph{Chevalley super-groups of classical type} \cite{FG}.
See also Grishkov and Zubkov \cite{GZ}, and Shibata \cite{Sh}. 
\end{rem}

Let $G$ be an affine algebraic super-group, and let $A= \mathscr{O}(G)$. Then $A$
includes the largest (purely even) separable subalgebra $\pi_0A$, which is necessarily a 
Hopf subalgebra. 
Let $\pi_0G$ denote the finite etale affine group represented by $\pi_0A$. Let $G^{\circ}$ denote
the affine algebraic closed super-subgroup of $G$ which is represented by the quotient Hopf super-algebra
$A/((\pi_0A)^+)$ of $A$; if $G=G_{\mathsf{ev}}$, this $G^{\circ}$ coincides with what appeared in Part 2 of
Theorem \ref{thm:Sullivan}, as $F^{\circ}$ for $F$.  
We have the short exact sequence $G^{\circ} \to G \to \pi_0G$ of affine algebraic super-groups. 
One sees from \cite[Section 2.2]{MZ2}
\begin{equation}\label{eq:piO}
\pi_0(G_{\mathsf{ev}})=\pi_0 G,\quad (G_{\mathsf{ev}})^{\circ}=(G^{\circ})_{\mathsf{ev}}, 
\end{equation}
and that the relevant constructions commute with base extension. We say that $G$ is \emph{connected}
if $G=G^{\circ}$, or equivalently, if $\pi_0G$ is trivial; by \eqref{eq:piO},
this is equivalent to saying that $G_{\mathsf{ev}}$ is connected.

\begin{lemma}\label{lem1}
The following are equivalent:
\begin{itemize}
\item[(a)] $G$ has an integral; 
\item[(b)] $G_{\bar{k}}$ has an integral;
\item[(c)] $G^{\circ}$ has an integral;
\item[(d)] $G^{\circ}_{\bar{k}}$ has an integral.
\end{itemize}
\end{lemma}
\begin{proof}
For (a) $\Leftrightarrow$ (b), modify the proof of \cite[Proposition 2.1]{Su2} 
into the super situation. 
We see that (a) $\Leftrightarrow$ (c) follows from Lemma \ref{lem2} below. 
The just proved equivalence applied to $G_{\bar{k}}$ gives (b) $\Leftrightarrow$ (d). 
\end{proof}

\begin{lemma}\label{lem2}
We have the following.
\begin{itemize}
\item[(1)] Every finite affine super-group has an integral. 
\item[(2)] Suppose that $N \to G \to Q$ is a short exact sequence of affine
super-groups. Then $G$ has an integral if and only if $N$ and $Q$ both have integrals. 
\end{itemize}
\end{lemma}
\begin{proof}
(1)\ This follows from Proposition \ref{prop:integral_boson}, by applying to 
$\mathscr{O}(G)\lboson \mathbb{Z}_2$ 
the well-known fact that every finite-dimensional Hopf algebra is co-Frobenius; see \cite[Sections 5.2-5.3]{DNR}. 

(2)\ The proof of \cite[Theorem 2.20]{Su2} in the non-super situation works well. 
\end{proof}

We are going to prove the remaining implication (ii) $\Rightarrow$ (i) of Theorem \ref{thm:integral}.
By Lemma \ref{lem1} we may and do assume that $k$ is algebraically closed, and $G$ is connected. 
Let us write $A=\mathscr{O}(G),\ H=\mathscr{O}(G_{\mathsf{ev}})$, as in \eqref{eq:A}, \eqref{eq:H}. 

\begin{rem}\label{rem:Scheunert_Zhang}
In characteristic zero, the results we are going to prove (assuming as above) are in part, quite similar
to those obtained by Scheunert and Zhang \cite{SZ}. 
An advantage of ours is an explicit formula (see Remark \ref{rem:explicit_integral}) of the integral for $G$;
it will be applied to prove the result, Proposition \ref{prop:unimodular}, 
which shows precisely when $G$ is unimodular (Definition \ref{def:G_with_integral}),
and thereby proves that $G$ is unimodular if $G_{\mathsf{ev}}$
is semisimple. 
\end{rem}

\subsubsection{Proof in positive characteristic}\label{subsubsec:proof_in_char_p}
Suppose $\operatorname{char} k=p>0$. 
In view of Theorem \ref{thm:Sullivan} (2),
we should prove that $G$ has an integral, assuming that $(G_{\mathsf{ev}})_{\mathsf{red}}$ is a torus.  
Given a positive integer $r$, the $r$-iterated Frobenius morphism gives the short exact sequence
with finite kernel
\[ G_r \to G \to G^{(r)}. \] 
Thus $\mathscr{O}(G^{(r)})$ is spanned by the elements $a^{p^r}$, where $a \in A$. We see from the 
isomorphism \eqref{eq:isom} that for $r$ large enough, $\mathscr{O}(G^{(r)})$ is purely even and reduced, and
is naturally embedded into $H/\sqrt{0}=\mathscr{O}((G_{\mathsf{ev}})_{\mathsf{red}})$. 
Hence $G^{(r)}$ is a torus, which has an integral. This together with Lemma \ref{lem2} prove
the desired result. 

\begin{rem}
Zubkov and Marko \cite{ZM} investigated closely the Frobenius kernels $G_r$,
when $G$ is a general linear super-group $GL(m|n)$. 
\end{rem}

\subsubsection{Proof in characteristic zero}\label{subsubsec:proof_in_char_zero}
Suppose $\operatorname{char}k = 0$. 
Assume that the connected affine algebraic group $G_{\mathsf{ev}}$ has an integral, or equivalently, $G_{\mathsf{ev}}$
is linearly reductive; this last is equivalent to saying that $G_{\mathsf{ev}}$ is reductive since it is now connected. 
For the vector space of comodule maps we will use the notation \eqref{eq:HomC}. 
Then we have $\operatorname{Hom}^{-H}(H,k)\ne 0$. 
It suffices to prove the following:

\begin{prop}\label{prop:integral_isom}
We have a $k$-linear isomorphism
\begin{equation}\label{eq:integral_isom}
\operatorname{Hom}^{-A}(A,k)\simeq \operatorname{Hom}^{-H}(H,k). 
\end{equation}
\end{prop}

We are going to prove this, showing
explicitly the isomorphism; see Remark \ref{rem:explicit_integral}. 

Let $\mathfrak{g}=\operatorname{Lie}(G)$ be the Lie super-algebra of $G$; see \cite[Section 3]{Z1}, for example.
This $\mathfrak{g}$ is finite-dimensional. 
Note that the odd component $\mathfrak{g}_1$ of $\mathfrak{g}$ coincides with the dual
vector space $W^*$ of the $W$ in \eqref{eq:isom}. 
The even component $\mathfrak{g}_0$ of $\mathfrak{g}$ coincides with the Lie algebra 
$\operatorname{Lie}(G_{\mathsf{ev}})$ of $G_{\mathsf{ev}}$, which is now reductive. 

The universal envelope
$U(\mathfrak{g})$ of $\mathfrak{g}$ 
is a Hopf super-algebra with all elements in $\mathfrak{g}$ primitive, 
which is super-cocommutative, but is not necessarily super-commutative; it includes
the universal envelope $U(\mathfrak{g}_0)$ of $\mathfrak{g}_0$ as the largest purely even
Hopf super-subalgebra.  
We have the canonical pairing $\langle \ , \ \rangle : U(\mathfrak{g})\times A \to k$. 
This defines on $A$, the natural left $U(\mathfrak{g})$-module structure 
\[
u \cdot a := \sum_{(a)} a_{(1)}\, \langle u, a_{(2)}\rangle,\quad
u \in U(\mathfrak{g}),\ a\in A.
\]
Here and in what follows, $\Delta(a)=\sum_{(a)} a_{(1)}\otimes a_{(2)}$ denotes the co-product on 
any Hopf (super-)algebra; in addition, see \cite[Section 3.2]{MS2} for sign convention. 
Similarly, the canonical pairing $U(\mathfrak{g}_0) \times H \to k$ defines
a natural left $U(\mathfrak{g}_0)$-module structure on $H$. For the algebra $R=U(\mathfrak{g})$ 
or $U(\mathfrak{g}_0)$, we let $\operatorname{Hom}_{R-}$ denote the
vector space of left $R$-module maps. 

\begin{lemma}\label{lem:identity}
We have
\begin{equation}\label{eq:identity}
\operatorname{Hom}^{-A}(A,k)=\operatorname{Hom}_{U(\mathfrak{g})-}(A,k). 
\end{equation}
\end{lemma}
\begin{proof}
Since $G_{\mathsf{ev}}$ is connected, the same argument as proving \cite[Proposition 20]{M2}, that uses \eqref{eq:isom} essentially,
shows that the natural Hopf super-algebra map from $A$ to the dual Hopf super-algebra of $U(\mathfrak{g})$,
which arises from the pairing $\langle \ , \ \rangle : U(\mathfrak{g})\times A \to k$ above, is injective. 
This implies the desired result. 
\end{proof}

It is proved by \cite[Proposition 4.25]{MS1} that the left $U(\mathfrak{g})$-module map
\begin{equation*}\label{eq:eta}
\eta : A \to \operatorname{Hom}_{U(\mathfrak{g}_0)-}(U(\mathfrak{g}), H) 
\end{equation*}
defined by
\begin{equation*}
\eta(a)(u)=\sum_{(a)}\pi(a_{(1)})\, \langle u,\ a_{(2)}\rangle,\quad a \in A,\ u \in U(\mathfrak{g}) 
\end{equation*}
is an isomorphism, where $\pi : A \to H=A/(A_1)$ is the natural projection.

Recall from \cite[Appendix]{SZ} the following argument, modifying it so as to work on the opposite side. 
Let $n=\dim \mathfrak{g}_1$, and choose arbitrarily a basis $x_1, x_2,\dots, x_n$ of $\mathfrak{g}_1$. 
Let $\Lambda_n$ be as in \eqref{eq:increasing_sequence}. 
Then $U(\mathfrak{g})$ has 
\begin{equation}\label{eq:free_basis}
x_I=x_{i_1}x_{i_2}\dots x_{i_r},\quad I=(i_1,i_2,\dots,i_r)\in \Lambda_n,
\end{equation}
as a left (and right) $U(\mathfrak{g}_0)$-free basis; we have $x_{\emptyset}=1$ by convention. 
Let $L=(1,2,\dots, n)$ denote the longest sequence in $\Lambda_n$. 
Let $\varpi: U(\mathfrak{g})\to U(\mathfrak{g}_0)$ be the map which assigns the special coefficient $c_L$
to every element $\sum_{I\in \Lambda_n}c_Ix_I$ in $U(\mathfrak{g})$, where $c_I\in U(\mathfrak{g}_0)$.
This is clearly left $U(\mathfrak{g}_0)$-linear. 
Given $x \in \mathfrak{g}_0$, let $\operatorname{ad}' \! x : \mathfrak{g}_1 \to \mathfrak{g}_1,\ 
(\operatorname{ad}' \! x)(y) = [x, y]$ denote the adjoint action on $\mathfrak{g}_1$. 
Let $\alpha : U(\mathfrak{g}_0) \to U(\mathfrak{g}_0)$ be the algebra automorphism determined by
\[ \alpha(x) = x + \operatorname{Tr}(\operatorname{ad}' \! x)1, \quad x \in \mathfrak{g}_0. \]
Let $\bar{\alpha}=\alpha^{-1}$ denote the inverse of $\alpha$. 
Then one sees that 
\[
\varpi(x_I\alpha(c))=\begin{cases}\, c & \text{if}~~I=L; \\ \, 0 & \text{otherwise}, \end{cases}
\]
where $I\in \Lambda_n$ and $c \in U(\mathfrak{g}_0)$.
(To see this,
one may replace $U(\mathfrak{g})$ with the naturally associated graded Hopf super-algebra $U(\mathfrak{g})_{\mathrm{gr}}$
as given in \cite[Section 3]{M1},
or in other words, one may suppose $[\mathfrak{g}_1,\mathfrak{g}_1]=0$.) It follows that
$\varpi :U(\mathfrak{g})\to U(\mathfrak{g}_0)_{\bar{\alpha}}$ is right $U(\mathfrak{g}_0)$-linear,
where $U(\mathfrak{g}_0)_{\bar{\alpha}}$ indicates the right $U(\mathfrak{g}_0)$-module obtained from 
$U(\mathfrak{g}_0)$ by twisting the action through $\bar{\alpha}$. 
The associated pairing
\begin{equation*}\label{eq:pairing} 
(\ , \ )_{\varpi} : U(\mathfrak{g}) \times U(\mathfrak{g})\to  U(\mathfrak{g}_0),\quad (u,v)_{\varpi}=\varpi(uv) 
\end{equation*}
makes $U(\mathfrak{g}) \supset U(\mathfrak{g}_0)$ into a \emph{free} 
$\alpha$-\emph{Frobenius extension}. This means
that given a left $U(\mathfrak{g}_0)$-module $M$, the left 
$U(\mathfrak{g})$-module map
\begin{equation}\label{eq:U-isom}
U(\mathfrak{g})\otimes_{U(\mathfrak{g}_0)} M \to
\operatorname{Hom}_{U(\mathfrak{g}_0)-}(U(\mathfrak{g}),\ {}_{\alpha}M), 
\end{equation}
which assigns to an element $v \otimes m\in U(\mathfrak{g})\otimes_{U(\mathfrak{g}_0)} M$, the left $U(\mathfrak{g}_0)$-module map
$U(\mathfrak{g})\to U(\mathfrak{g}_0)_{\bar{\alpha}}\otimes_{U(\mathfrak{g}_0)}M= {}_{\alpha}M$, $u \mapsto \alpha((u,v)_{\varpi})m$,
is an isomorphism. 
Here ${}_{\alpha}M$ indicates the twisted left $U(\mathfrak{g}_0)$-module, as before. 

Take as the $M$ in \eqref{eq:U-isom}, the twisted left $U(\mathfrak{g}_0)$-module ${}_{\bar{\alpha}}H$. 
Then the resulting isomorphism composed with $\eta$ gives an isomorphism 
\begin{equation*}\label{eq:A_to_UH}
A \simeq U(\mathfrak{g}) \otimes_{U(\mathfrak{g}_0)} {}_{\bar{\alpha}}H 
\end{equation*}
of left $U(\mathfrak{g})$-modules. 
Note that there uniquely exists a right $U(\mathfrak{g}_0)$-free basis $y_I, I\in \Lambda_n$, of $U(\mathfrak{g})$
which is \emph{dual} to the basis $x_I$ in \eqref{eq:free_basis} with respect to the pairing $(\ , \ )_{\varpi}$ 
in the sense that $(x_I, y_J)_{\varpi}=\delta_{I,J}$, the Kronecker delta. Then one sees that the last isomorphism 
is explicitly given by
\begin{equation}\label{eq:explicit_isom}
A\overset{\simeq}{\longrightarrow} U(\mathfrak{g})\otimes_{U(\mathfrak{g}_0)}{}_{\bar{\alpha}}H,\quad 
a\mapsto \sum_{(a), I\in \Lambda_n} y_I\otimes \pi ( a_{(1)})\, \langle x_I,\ a_{(2)}\rangle. 
\end{equation}
In view of the equation \eqref{eq:identity} and an analogous one, we have
\begin{equation*}
\operatorname{Hom}^{-A}(A, k) \simeq \operatorname{Hom}_{U(\mathfrak{g}_0)-}({}_{\bar{\alpha}}H, k )
= \operatorname{Hom}^{-H}({}_{\bar{\alpha}}H, k ).
\end{equation*}
It remains to prove that ${}_{\bar{\alpha}}H\simeq H$ or $H \simeq {}_{\alpha}H$ as right $H$-comodules.

Let $\delta : U(\mathfrak{g}_0) \to k$ be the algebra map determined by
\begin{equation}\label{eq:delta}
\delta(x) = \operatorname{Tr}(\operatorname{ad}' \! x),\quad x \in \mathfrak{g}_0;
\end{equation}
see \cite[(5.2)]{SZ}. 
The assumption $k=\bar{k}$ ensures that $G_{\mathsf{ev}}$ includes a split maximal torus, say, $T$. 
Let $\mathfrak{h}\subset \mathfrak{g}_0$ be
the corresponding Cartan subalgebra. 
Since the restriction $\delta|_{\mathfrak{h}}$ of $\delta$ to $\mathfrak{h}$ 
is an element in the character group $X(T)$ of $T$, it follows that the one-dimensional $U(\mathfrak{g}_0)$-module
structure given by $\delta$ arises from an $H$-comodule structure, whence $\delta$ is 
a grouplike in $H\, (\subset U(\mathfrak{g}_0)^*)$; see \cite[Part II, 1.20]{J}.
Since one sees that $\alpha$ coincides with 
$u \mapsto \sum_{(u)} u_{(1)}\, \delta(u_{(2)})$,
the left $U(\mathfrak{g}_0)$-module structure on 
${}_{\alpha}H$ arises from the right $H$-comodule structure
\[ {}_{\alpha}H \to {}_{\alpha}H \otimes H,\ \quad h 
\mapsto \sum_{(h)} h_{(1)} \otimes h_{(2)}\delta. \]
We see that $h \mapsto h \, \delta$ gives a desired isomorphism ${}_{\alpha}H \simeq H$, 
since $\delta$ is grouplike. 
This completes the proofs of Proposition \ref{prop:integral_isom} and of Theorem \ref{thm:integral}. 

\begin{rem}\label{rem:explicit_integral}
In view of \eqref{eq:explicit_isom}, we see that
the obtained isomorphism assigns 
to each $\psi$ in $\operatorname{Hom}^{-H}(H,k)$, the right $A$-comodule map
\begin{equation}\label{eq:explicit_integral}
\phi: A \to k,\quad \phi(a)=\sum_{(a)}\psi(\pi(a_{(1)})\delta^{-1})\, \langle z, \ a_{(2)} \rangle
\end{equation}
(cf.~\cite[(4.2)]{SZ}),\ where we set
\begin{equation*}\label{eq:z}
z:=\sum_{I\in \Lambda_n}x_I\, \varepsilon(y_I).
\end{equation*}
\end{rem}

Suppose that $k$ is a field of characteristic zero which may not be algebraically closed. 
Note that the definition \eqref{eq:delta} of the algebra map $\delta : U(\mathfrak{g}_0)\to k$
still makes sense. 

\begin{prop}[\text{cf.~\cite[Corollary~5.6]{SZ}}]\label{prop:unimodular}
Let $G$ be a quasi-reductive super-group \textup{(Remark \ref{rem:large_enough})} over the filed $k$ as above, 
or namely, a connected affine algebraic super-group over $k$ such that
$G_{\mathsf{ev}}$ is \textup{(}linearly\textup{)} reductive.  
\begin{itemize}
\item[(1)]
$G$ is unimodular if and only if
$\delta$ is trivial, or explicitly, $\delta(x)=0$, $x \in \mathfrak{g}_0$.
\item[(2)]
$G$ is unimodular if $G_{\mathsf{ev}}$ is semisimple.
\end{itemize}
\end{prop}
\begin{proof}
(1)\ By base extension we may suppose $k=\bar{k}$, so that $\delta$ is a grouplike in $H$. 
Since $G_{\mathsf{ev}}$ is linearly reductive, we can choose a left and right integral $\psi$ such that $\psi(1)=1$
(Remark \ref{rem:Weissauer}), to which corresponds 
a non-zero right integral $\phi$ for $G$ given by the formula \eqref{eq:explicit_integral}. 

In general, the uniqueness on integrals gives rise to the so-called \emph{distinguished grouplike} \cite[p.197]{DNR}.
In the present situation it is the homogeneous (necessarily, even) grouplike $\gamma \in A$ which satisfies
$\gamma \, \phi(a)= \sum_{(a)}a_{(1)}\phi(a_{(2)})$, or more explicitly,
\begin{equation}\label{eq:gz}
\gamma\, \big< z, \ \sum_{(a)} \psi(\pi(a_{(1)})\delta^{-1})\, a_{(2)}\big> =
\sum_{(a)}a_{(1)}\big< z, \ \psi(\pi(a_{(2)})\delta^{-1})\, a_{(3)}\big>
\end{equation}
for all $a \in A$; see also \cite[Proposition 1]{SZ0}, \cite[Proposition 2.2]{SZ}. 
This $\gamma$ equals the identity element if and only if $\phi$ is a \emph{left} integral, as well. 
Therefore, we should prove that $\gamma=1$ if and only if $\delta=1$. 
By \cite[Proposition 4.6~(3)]{M1}, any inclusion $B \hookrightarrow A$ of a Hopf super-subalgebra induces an injection $B/(B_1) \to A/(A_1)=H$
of Hopf algebras. This, applied to the Hopf super-subalgebra $B=k\Gamma$ spanned by all even grouplikes $\Gamma$,
tells us that we have only to prove $\pi(\gamma)=\delta$, 
since the induced injection is then $\pi|_{k\Gamma} : (B/(B_1)=)\, k\Gamma \to H$. 
Apply $\pi$ to both sides of \eqref{eq:gz}.
Then the desired equality follows
by using the equation
\[
\sum_{(a)}\pi(a_{(1)})\, \psi(\pi(a_{(2)})\delta^{-1})=\delta \, \psi(\pi(a)\delta^{-1}),
\] 
which holds since $\psi$ is a left integral . 

(2)\
If $G_{\mathsf{ev}}$ (or $\mathfrak{g}_0$) is semisimple, or equivalently, if $\mathfrak{g}_0=[\mathfrak{g}_0,\mathfrak{g}_0]$,
then $\delta$ is necessarily trivial. Therefore, $G$ is unimodular by Part 1.
\end{proof}

\subsection{Quotients by affine algebraic super-groups with integral}\label{subsec:affinity_integral}
Return to the situation where $k$ is an arbitrary field.
We now come to prove the following main result of ours.

\begin{theorem}\label{thm:affinity_integral}
Suppose that an affine algebraic super-group $G$ with integral acts 
strongly freely on 
an affine super-scheme $X$. 
Then the dur sheaf $X \tilde{\tilde{/}} G$ is the affine super-scheme represented by 
$\cO(X)^{G}$, and $X\to X \tilde{\tilde{/}} G$ is a $G$-super-torsor.
In addition, if $X$ is Noetherian, then the affine super-scheme $X \tilde{\tilde{/}} G$
is Noetherian, and it coincides with $X \tilde{/} G$. 
\end{theorem}
\begin{proof}
Let us apply the argument of Section \ref{subsec:boson_tech}, using the same notation. 
It suffices to prove 
\begin{itemize}
\item[(1)]
$\hat{\beta} : \hat{B}\otimes_C \hat{B}\to \hat{B}\otimes \hat{A}$
is bijective, and 
\item[(2)]
$B$ is a projective generator of $C$-modules. 
\end{itemize}
Note that $\hat{\beta}$ is surjective, and $\hat{A}$ is co-Frobenius by Proposition \ref{prop:integral_boson}. 
Corollary 3.3 of \cite{SS} (in Case (4) applied when $H=Q$) ensures (1)
and that $\hat{B}$ is projective as a left $C$-module. Since the latter implies that
$B$ is projective as a $C$-module, it remains to show that it is a generator.
We are going to use:
\begin{itemize}
\item[(3)]
We have the isomorphism $A\simeq \wedge(W)\otimes H$ as in \eqref{eq:isom}, 
in which $W$ and so $\wedge(W)$ are now finite-dimensional;
\item[(4)]
The isomorphism
$\beta : B \otimes_C B \to B \otimes A$ is $G$-equivariant, where
$G$ acts (or $A$ co-acts) on the right tensor-factors in $B \otimes_C B$ and in $B\otimes A$.   
\end{itemize}

Suppose $\op{char}k=0$. By Theorems \ref{thm:integral} and \ref{thm:Sullivan}, 
$G_{\mathsf{ev}}$ is linearly reductive.  
Note that the restricted action by $G_{\mathsf{ev}}$ on $X$ is strongly free. Let $D=B^{G_{\mathsf{ev}}}$ 
be the super-subalgebra of $G_{\mathsf{ev}}$-invariants in $B$, which clearly includes $C$. 
By \cite[Theorem 4.10]{SS}
applied to the right comodule algebra $\hat{B}$ over the cosemisimple
Hopf algebra $H \otimes k \mathbb{Z}_2$, we see that $\hat{B}$ and so $B$ are
(projective) generators of left (and right) $D$-modules. Hence $D$ is a projective
$C$-module. It remains to prove that $D$ is a generator of $C$-modules. 
We aim to prove that $D$ is finitely generated as a $C$-module;
this will imply the desired result by Lemma \ref{lem:generator}.

Consider invariants by the restricted $G_{\mathsf{ev}}$-action.
One sees from (3) that $R:=A^{G_{\mathsf{ev}}}$ is isomorphic to $\wedge(W)$, and $(B\otimes A)^{G_{\mathsf{ev}}}=
B\otimes R$. Since $B$ is flat as a right $C$-module, we have $(B\otimes_C B)^{G_{\mathsf{ev}}}=B\otimes_C D$,
whence $\beta$ restricts to the isomorphism
\[ B \otimes_C D \overset{\simeq}{\longrightarrow} B \otimes R. \]
Hence the projective left $C$-module $D$ is finitely generated after the base extension to $B$.
The result we aim at follows by Lemma \ref{lem:finitely_generated} below.

Suppose $\op{char}k>0$, and let $F=G_{\mathsf{ev}}$. 
By base extension we may suppose $k =\bar{k}$. Then we have the split short 
exact sequence $F^{\circ}\to F\to \pi_0F$ of affine algebraic groups. By Theorems 
\ref{thm:integral} and \ref{thm:Sullivan}, $T:=F^{\circ}_{\mathsf{red}}$ is a torus. Since
$T$ is smooth, it follows by \cite[Proposition 1.10]{MO} that the closed embedding $T \hookrightarrow F^{\circ}$
of right $T$-equivariant affine schemes splits.  
We have, therefore, an isomorphism
\[ H \simeq \mathscr{O}(\pi_0F) \otimes \mathscr{O}(F^{\circ})^T \otimes \mathscr{O}(T) \]
of right $\mathscr{O}(T)$-comodule algebras, in which $\mathscr{O}(\pi_0F)$
and $\mathscr{O}(F^{\circ})^T$ are both finite-dimensional;
the former is separable while the latter is local. 
Since $T$ is linearly reductive, we can modify the proof in characteristic zero, replacing 
$G_{\mathsf{ev}}$ with $T$, to obtain the result we aim at. 
\end{proof}

\begin{lemma}\label{lem:finitely_generated}
Let $C \subset B$ be an inclusion of non-commutative rings. A projective left $C$-module $P$
is finitely generated if $B \otimes_CP$ is a finitely generated left $B$-module.
\end{lemma}
\begin{proof}
By the projectivity $P$ is included as a direct summand in a free left $C$-module $F=\bigoplus_{i \in I}C_i$,
where $C_i=C$. It suffices to prove that there exists a finite subset $I_0 \subset I$
such that $P$ is included in $\bigoplus_{i \in I_0}C_i$, since the last direct sum then projects
onto $P$. By the assumption there exists a finite subset $I_0 \subset I$ such that
$B \otimes_C P$ is included in $\bigoplus_{i \in I_0}B_i\, (\subset B \otimes_C F)$, where $B_i=B\otimes_C C_i$. 
Hence $P$ is included in $F \cap \bigoplus_{i \in I_0}B_i=\bigoplus_{i \in I_0}C_i$. 
\end{proof}

\subsection{Remarks and an example}\label{subsec:example}
We remark that the subsequent preprint \cite{MOe} by Oe and the first named author 
discusses super-torsors, and contains a result, Theorem 1.8, which is reformulated as
follows, restricted to the affinity situation.

\begin{theorem}\label{thm:MOe}
Suppose that $k$ is a field of characteristic $\ne 2$. Suppose that an affine algebraic super-group $G$ acts strongly freely
on a Noetherian affine super-scheme $X$. Assume the following:
\begin{itemize}
\item[(a)] $G$ and $X$ are both smooth \cite[Section A.1]{MZ2}; 
\item[(b)] The dur sheaf $X_{\mathsf{ev}}\tilde{\tilde{/}}G_{\mathsf{ev}}$ associated with the
induced \textup{(}necessarily, strongly free\textup{)} action by the affine algebraic group $G_{\mathsf{ev}}$ on the Noetherian affine scheme $X_{\mathsf{ev}}$ 
is affine; it then necessarily follows that $X_{\mathsf{ev}}\tilde{\tilde{/}}G_{\mathsf{ev}}$ coincides with
$X_{\mathsf{ev}}\tilde{/}G_{\mathsf{ev}}$, and is Noetherian and smooth. 
\end{itemize}
Then the dur sheaf $X\tilde{\tilde{/}}G$ is a Noetherian smooth affine super-scheme, and coincides with the sheaf  
$X\tilde{/}G$. Moreover, 
$X \to X \tilde{\tilde{/}}G\, (=X\tilde{/}G)$ is a $G$-super-torsor, and 
we have 
\begin{equation}\label{eq:X/G}
\textup{(i)}\ \, (X \tilde{/} G)_{\mathsf{ev}}=X_{\mathsf{ev}}\tilde{/}G_{\mathsf{ev}}; \quad 
\textup{(ii)}\ \, X\simeq \big((X \tilde{/} G)\times_{(X\tilde{/}G)_{\mathsf{ev}}}X_{\mathsf{ev}}\big)\times^{G_{\mathsf{ev}}}G,
\end{equation}
where \textup{(ii)} is an isomorphism of right $G$-equivariant super-schemes 
over $X\tilde{/}G$.
\end{theorem}

Let us be in the situation of the theorem above, to restate the conclusions (i) and (ii) in the Hopf-algebra language. 
Suppose $\cO(G)=A$, $\cO(G_{\mathsf{ev}})=H$
and $\cO(X)=B$, as before. 
The affine super-scheme $X_{\mathsf{ev}}$ associated with $X$ is defined in the same way as the $G_{\mathsf{ev}}$ in 
Section \ref{subsec:preliminary_integral}, so that we have $\cO(X_{\mathsf{ev}})=B/(B_1)$. 
We now have  $\cO(X\tilde{/}G)=B^G$, and this is proved to be a Noetherian smooth super-algebra. 
Let $C=B^G$, as before, and define $\bar{B}:=B/(B_1)$,\ $\bar{C}:=C/(C_1)$. 
Then the conclusion (i)  means $\bar{C}=\bar{B}^{G_{\mathsf{ev}}}$. 
By the Noetherian smoothness of $C$, the canonical projection $C\to \bar{C}$ splits. Moreover,
if one chooses arbitrarily a section $\bar{C} \to C$, and regards $C$ as a super-algebra over $\bar{C}$ through the section,
then $C$ is isomorphic to the exterior algebra $\wedge_{\bar{C}}(P)$, where $P$ is a finitely projective $\bar{C}$-module;
see \cite[Theorem A.2]{MZ2}. The conclusion (ii) is restated as an isomorphism
\begin{equation}\label{eq:BCA}
B\ \simeq\ ( C\otimes_{\bar{C}} \bar{B} ) \square_HA\ (=C\otimes_{\bar{C}} (\bar{B}\square_HA ) ) 
\end{equation}
of right $A$-super-comodule algebras over $C$, by using the co-tensor product \cite[Section 2.3]{DNR}.

\begin{rem}\label{rem:compare}
To compare the theorem above with our Theorem \ref{thm:affinity_integral}, 
suppose that $k$ is a field of characteristic $\ne 2$, and  
let $G$ be an affine algebraic super-group. Recall that if $\operatorname{char} k=0$, then $G$
is necessarily smooth; see \cite[Proposition A.3]{MZ2}. 
Assume that $G$ has an integral, and it 
acts strongly freely on an affine super-scheme $X$. The theorem above
shows that the dur sheaf $X\tilde{\tilde{/}}G$ is an affine super-scheme,
and $X \to X\tilde{\tilde{/}}G$ is a $G$-super-torsor, under the assumptions:
(1$^{\circ}$)~$X$ is Noetherian and smooth, and (2$^{\circ}$)~if $\operatorname{char} k>2$, then $G$ is smooth, in addition. 
This is because the assumption (b) of the theorem is then satisfied; in characteristic zero, the assumption is ensured
by Theorem \ref{thm:known} in Case (ii),
since $G_{\mathsf{ev}}$ is linearly reductive by Sullivan's Theorem \ref{thm:Sullivan} (1).  
We remark that in positive characteristic, the assumption (2$^{\circ}$) is satisfied 
if and only if $G_{\mathsf{ev}}^{\circ}$ is a torus, as is seen from Sullivan's Theorem \ref{thm:Sullivan} (2). 
In this case it was probably known that
the assumption (b) is satisfied; at least now, this (b) is
ensured by our Theorem \ref{thm:affinity_integral}
applied in the non-super situation.
Important is the fact that this Theorem \ref{thm:affinity_integral} ensures 
without the assumptions (1$^{\circ}$), (2$^{\circ}$)
the conclusions that $X\tilde{\tilde{/}}G$ is an affine 
super-scheme, and $X \to X\tilde{\tilde{/}}G$ is a $G$-super-torsor. 
\end{rem}

Recall from \eqref{eq:X/G} (ii) or \eqref{eq:BCA} 
that under the assumptions (1$^{\circ}$), (2$^{\circ}$), $X$ is of a restricted form. Therefore,
one may conclude that our Theorem \ref{thm:affinity_integral} is rather of theoretical importance; it is expected
to bring some applications to some subject related to super-torsors. 
Highly expected is to play a role in generalizing
Picard-Vessiot Theory to the super context.

The preprint \cite{MOe} contains another result, Theorem 1.3, proved
under the affinity assumption, which is
reformulated, to state it briefly, as follows. 

\begin{theorem}
Suppose that $k$ is as in {Theorem \ref{thm:MOe}}, and let $G$ be a smooth affine algebraic super-group. 
Given a Noetherian smooth affine super-scheme $Y$, every $G$-super-torsor $X \to Y$ 
arises 
from an ordinary $G_{\mathsf{ev}}$-torsor $E \to Y_{\mathsf{ev}}$, so that
\begin{equation}\label{eq:XYE}
X=(Y \times_{Y_{\mathsf{ev}}} E)\times^{G_{\mathsf{ev}}}G.
\end{equation}
\end{theorem}

Here is a simple example which shows that for this result, the smoothness assumption 
for $G$ is indeed needed 
in $\operatorname{char}k>2$. 

\begin{example}\label{ex:alpha}
Suppose $p=\operatorname{char} k>2$, choose an integer $r>0$ and let $q=p^r$. Let $G=\alpha_q$
be the affine group of elements whose $q$-th powers are zero; it is
represented by the Hopf algebra $A=k[t]/(t^q)$ in which $t$ is primitive. 
Obviously, this $G$ is not smooth, while it, being finite, has an integral. 
Choose arbitrarily a super-algebra $C\ne 0$ and an even element $\tau \in C_0$.  
Let $Y$ denote the affine super-scheme represented by $C$. 
Define a super-algebra over $C$ by
\[
B_\tau:=C[T]/(T^q-\tau),\quad |T|=0
\]
and let $X_\tau$ denote the affine super-scheme represented by this $B_\tau$. 
We can make $B_\tau$ into a right $A$-super-comodule algebra by
\[ 
\rho: B_\tau \to B_\tau \otimes A,\quad \rho(T)=T\otimes 1+1 \otimes t, 
\]
so that 
the associated action by $G$ on $X_\tau$ is 
strongly free, and $X_\tau\tilde{\tilde{/}}G=Y$, 
whence $X_{\tau} \to Y$ is a $G$-super-torsor. 
Assume $C/(C_1)=k$, and choose $\tau$ so as
\begin{equation*}\label{eq:tau}
\text{(i)}\ \tau \equiv 0 \operatorname{mod}(C_1),\quad \text{(ii)}~~\tau\ne c^q~~\text{for any}~~c\in C_0. 
\end{equation*}
Then $X_\tau$ is not of the form \eqref{eq:XYE}; by (i), this means that
$B_{\tau}$ is not isomorphic to $C\otimes A$ as a right $A$-super-comodule
algebra over $C$. In fact, if there were an isomorphism $\theta : C\otimes A \to B_{\tau}$, then we would have $T-\theta(1\otimes T)=c$ for some $c\in C_0$, so that
$\tau=T^q-\theta(1\otimes T)^q=c^q$, which contradicts (ii).

For example,
if $C$ is the exterior algebra
$\wedge(W)$ on a finite-dimensional vector space $W$ of dimension $>1$, then $C/(C_1)=k$,
$Y$ is Noetherian and smooth, and one can choose any non-zero element in 
$\bigoplus_{i>0}\wedge^{2i}(W)$
as the $\tau$ which satisfies (i), (ii).
\end{example}

\section*{Acknowledgments}
The first- and the second-named authors were supported by JSPS~KAKENHI, Grant Numbers 17K05189 and 19K14517, respectively.

\end{document}